\documentclass[11pt]{amsart}

\usepackage{amsmath,graphicx,amssymb,amsthm,amstext}
\usepackage{caption}
\usepackage{subcaption}
\usepackage{tikz}
\usepackage[ruled]{algorithm2e}
\usepackage{listings}
\usepackage{hyperref}

\newtheorem{theorem}{Theorem}[section]
\newtheorem{lemma}[theorem]{Lemma}

\theoremstyle{definition}

\theoremstyle{remark}

\title{An explicit edge-coloring of $K_n$ with six colors on every $K_5$}
\author{Alex Cameron}
\address[Alex Cameron]{Department of Mathematics, Statistics, and Computer Science, University of Illinois at Chicago}
\email{acamer4@uic.edu}
\date{}

\begin{document}
\maketitle

\begin{abstract}
For fixed integers $p$ and $q$, let $f(n,p,q)$ denote the minimum number of colors needed to color all of the edges of the complete graph $K_n$ such that no clique of $p$ vertices spans fewer than $q$ distinct colors. A construction is given which shows that $f(n,5,6) \leq n^{1/2+o(1)}$. This improves upon the best known probabilistic upper bound of $O\left(n^{3/5}\right)$ given by Erd\H{o}s and Gy{\'a}rf{\'a}s. It is also shown that $f(n,5,6) = \Omega\left(n^{1/2}\right)$.
\end{abstract}

\section{Introduction}

Let $K_n$ denote the complete graph on $n$ vertices. Fix positive integers $p$ and $q$ such that $1 \leq q \leq {p \choose 2}$. A $(p,q)$-coloring of $K_n$ is any edge-coloring such that every copy of $K_p$ contains edges of at least $q$ distinct colors. Let $f(n,p,q)$ denote the minimum number of colors needed to give a $(p,q)$-coloring of $K_n$.

Paul Erd\H{o}s originally introduced the function $f(n,p,q)$ in 1981 \cite{erdos1981}, but it was not studied systematically until 1997 when Erd\H{o}s and Andr{\'a}s Gy{\'a}rf{\'a}s \cite{erdos1997} looked at the growth rate of $f(n,p,q)$ as $n \rightarrow \infty$ for fixed values of $p$ and $q$. In this paper, they used the Lov{\'a}s Local Lemma to give a general upper bound, \[f(n,p,q) = O\left(n^{\frac{p-2}{1-q+{p \choose 2}}}\right).\]

Since then, this function has been studied for various values of $p$ and $q$. Here we improve the probabilistic upper bound of $f(n,5,6)$ by giving an explicit $(5,6)$-coloring of $K_n$ that uses few colors. The new upper bound comes close to matching the lower bound in order of magnitude.

\begin{theorem}
\label{main}
As $n \rightarrow \infty$, \[\left( \frac{5}{6}n - \frac{95}{144} \right)^{1/2} \leq f(n,5,6) \leq n^{1/2}2^{O \left(\sqrt{\log n}\log \log n \right)}.\]
\end{theorem}

The lower bound comes from the following lemma, a generalization of an argument used by Erd\H{o}s and Gy\'arf\'as \cite{erdos1997} and stated explicitly as equation 11 in \cite{CFLS}.

\begin{lemma}
Let $t=f(n,p,q)$, then \[f \left(\left\lceil \frac{n-1}{t} \right\rceil,p-1,q-1 \right) \leq t.\]
\end{lemma}

\begin{proof}
Suppose we have a $(p,q)$-coloring of $K_n$ with $t$ colors. Fix some vertex $x$, then at least $\left\lceil \frac{n-1}{t} \right\rceil$ vertices must appear in a monochromatic neighborhood of $x$. The number of colors $t$ must be enough to give a $(p-1,q-1)$-coloring on this set.
\end{proof}

Erd\H{o}s and Gy\'{a}rf\'{a}s showed that $\frac{5}{6}(n-1) \leq f(n,4,5)$ \cite{erdos1997}. This, combined with the lemma, gives the stated lower bound in Theorem~\ref{main}.

The construction providing the upper bound combines two existing constructions with some modification. The first was given recently by Conlon, Fox, Lee, and Sudakov \cite{CFLS} and was originally used to show \[f(n,p,p-1) \leq 2^{16p(\log n)^{1-1/(p-2)}\log \log n}.\] The second construction is the ``algebraic" part of the $(4,4)$-coloring given by Mubayi in \cite{mubayi2004}.

After defining the construction in the next section, we will demonstrate that it avoids many different configurations of colored edges on five or fewer vertices. By ruling these cases out, a simple algorithm is used to show that no copy of $K_5$ can span fewer than six distinct colors.

\section{The Construction}

The construction is the product of two separate edge-colorings, $\varphi \times \chi$, where $\varphi$ is a modified version of a particular instance of the general construction given by Conlon, Fox, Lee, and Sudakov (CFLS) \cite{CFLS} and $\chi$ is a modified version the the algebraic portion of the construction given by Mubayi \cite{mubayi2004}.

\subsection{The Modified CFLS Coloring}

The Modified CFLS coloring will be a product of two colorings, $\varphi = \varphi_1 \times \varphi_2$. The first, $\varphi_1$ is a particular simple case of the more general $(p,p-1)$-coloring given in \cite{CFLS}. The modification, $\varphi_2$, was originally given by this author and Emily Heath for our recent $(5,5)$-coloring \cite{55}.

Let $n = 2^{\beta^2}$ for some positive integer $\beta$. Associate each vertex of $K_n$ with a unique binary string of length $\beta^2$. That is, we may assume that our vertex set is \[V = \{0,1\}^{\beta^2}.\] For any vertex $v \in V$, let $v^{(i)}$ denote the $i$th block of bits of length $\beta$ in $v$ so that \[v = (v^{(1)},\ldots,v^{(\beta)})\] where each $v^{(i)} \in \{0,1\}^{\beta}$.

Between two vertices $x,y \in V$, the CFLS coloring is defined by \[\varphi_1(x,y) = \left( \left( i, \{x^{(i)},y^{(i)}\}\right), i_1, \ldots, i_{\beta}\right)\] where $i$ is the first index for which $x^{(i)} \neq y^{(i)}$, and for each $k=1,\ldots, \beta$, $i_k = 0$ if $x^{(k)} = y^{(k)}$ and otherwise is the first index at which a bit of $x^{(k)}$ differs from the corresponding bit in $y^{(k)}$.

For convenience, when discussing any edge color $\alpha$, we will let $\alpha_0$ denote the first coordinate of the color (of the form $(i,\{x^{(i)},y^{(i)}\})$) and let $\alpha_k$ denote the index of the first bit difference of the $k$th block for $k=1,\ldots,\beta$.

We will define the coloring $\varphi_2$ on the same set of vertices as the CFLS coloring, $V=\{0,1\}^{\beta^2}$. However, we will also need to consider the vertices as an ordered set. Consider each vertex to be an integer represented in binary. Then order the vertices by the standard ordering of the integers. That is, $x<y$ if and only if the first bit at which $x$ and $y$ differ is zero in $x$ and one in $y$.

Note that each $\beta$-block is a binary representation of an integer from $0$ to $2^{\beta}-1$,  so these blocks can be considered ordered in the same way. Moreover, note that if $x<y$ and if the first $\beta$-block at which $x$ and $y$ differ is $i$, then it must be the case that $x^{(i)}<y^{(i)}$.

Let $x,y \in V$ such that $x<y$. We define the second coloring as \[\varphi_2(x,y) = \left( \delta_1(x,y),\ldots,\delta_{\beta}(x,y) \right)\] where for each $i$, \[\delta_i(x,y) = \left\{
        \begin{array}{ll}
            -1 & \quad x^{(i)} > y^{(i)} \\
            +1 & \quad x^{(i)} \leq y^{(i)}
        \end{array}
    \right. .\]
    
This construction uses $2^{\beta}$ colors. Therefore, the modified CFLS coloring, $\varphi = \varphi_1 \times \varphi_2$, uses \[\beta^{\beta+1}2^{3\beta} = \sqrt{\log n}^{\sqrt{\log n}+1}2^{3\sqrt{\log n}} = 2^{O \left( \sqrt{\log n} \log \log n\right)}\] colors.

\subsection{The Modified Algebraic Coloring}

We will now define the algebraic part of the construction, $\chi = \chi_1 \times \chi_2$. The first part of this construction, $\chi_1$, is exactly the algebraic part of the $(4,4)$-coloring given by Mubayi \cite{mubayi2004}. The second part, $\chi_2$, is a modification original to this paper but based on a similar modification used to alter the algebraic portion of the $(5,5)$-coloring in \cite{55}.

Let $n = q^2$ where $q$ is some odd prime power. Associate each vertex of $K_n$ with a unique vector in the space $\mathbb{F}_q^2$ over the finite field with $q$ elements. Between any two vectors $x=(x_1,x_2)$ and $y=(y_1,y_2)$, we define the color $\chi_1$ of the edge between them as \[ \chi_1(xy) = \left(x_1y_1 - x_2 - y_2, \delta(x_1,y_1) \right) \] where \[\delta(x_1,y_1) = \left\{
        \begin{array}{ll}
            0 & \quad x_1 = y_1 \\
            1 & \quad x_1 \neq y_1
        \end{array}
    \right. .\]
Here, all algebraic operations are taken to be the standard ones defined by the finite field.
    
The modification to this coloring, $\chi_2$, requires that we give the elements of $\mathbb{F}_q^2$ some linear order. When we combine the algebraic part of the coloring with the modified CFLS coloring, this order will agree with the order put on the binary strings, but for now we just assume that there is some linear order.

For each element $\alpha \in \mathbb{F}_q$ let $G_{\alpha}$ be the graph with vertex set $\mathbb{F}_q \setminus \{\alpha\}$ such that \[xy \in E(G) \iff x+y = 2 \alpha.\] It is straightforward to show that the edges of $G_{\alpha}$ form a complete matching of the vertices. The vertices can therefore be partitioned into two sets, $S_{\alpha}$ and $T_{\alpha}$, such that no edge lies inside either set.

For two distinct elements, $\alpha, \beta \in \mathbb{F}_q$, define the function \[f_{\alpha}(\beta) = \left\{
        \begin{array}{ll}
            S & \quad \beta \in S_{\alpha} \\
            T & \quad \beta \in T_{\alpha}
        \end{array}
    \right. .\]
Now we can define $\chi_2$ for two vectors, $x < y$, as \[ \chi_2(xy) = \left(f_{x_1}(y_1),f_{y_1}(x_1)\right).\]

The coloring, $\chi_1$, gives at most $2q$ colors on $q^2$ vertices, and the modification, $\chi_2$, gives four colors. So overall the modified algebraic coloring $\chi$ uses at most $8q = 8 \sqrt{n}$ colors.

\subsection{Combining the Constructions}

Begin with $n=q^2$ for some odd prime power $q$, and associate each vertex with a distinct vector of $\mathbb{F}_q^2$ as in the previous section. Give some linear order for the elements of the base field, $\mathbb{F}_q$. To each $\alpha \in \mathbb{F}_q$ we associate the unique element $\alpha' \in \{0,1\}^{\left\lceil \log q \right\rceil}$ which represents in binary the rank of $\alpha$ under the this linear order.

Let $\beta$ be the minimum positive integer for which \[2\left\lceil \log q \right\rceil \leq \beta^2. \] To a vertex of $K_n$ associated with vector $(x_1,x_2) \in \mathbb{F}_q^2$, we also associate the binary string $(x_1',x_2',0) \in \{0,1\}^{\beta^2}$ where for each $i$, $x_i'$ is the binary representation of the rank of $x_i$, and $0$ denotes a string of $\beta^2 -2 \left\lceil \log q \right\rceil$ zeros.

The edge-coloring $\varphi \times \chi$ is then given by applying $\chi$ to the vectors and $\varphi$ to the binary strings. Since \[\beta = \Theta \left( \sqrt{2\log q} \right) =  \Theta \left( \sqrt{\log n} \right),\]  it follows that the number of colors used in this combined coloring is at most \[8q \beta 2^{2 \beta} = n^{1/2} 2^{O\left(\sqrt{\log n} \log \log n \right)}\] colors. This upper bound on the number of colors generalizes to all $n$ by the standard density of primes argument \cite{mubayi2004, perelli1984}.

\section{Configurations Avoided by CFLS}

In \cite{55}, we showed that the modified CFLS coloring, $\varphi$, avoids several possible configurations of edge colors on small cliques. Several of these cases, including monochromatic odd cycles, are covered by Lemma~\ref{colorcycle}.

\subsection{General ``color cycle" configurations}

Let $p$ and $q$ be positive integers. Assume that we have a copy of $K_p$ under an edge-coloring \[c:E(K_p) \rightarrow \{C_1,\ldots,C_q\}.\] Define an auxiliary digraph $D$ on the set of edge colors, $V(D) = \{C_1,C_2,\ldots,C_q\}$, such that $C_i \rightarrow C_j \in E(D)$ if and only if there exist vertices $v_1,\ldots,v_k \in V(K_p)$ for $k \geq 3$ such that \[c(v_1v_2) =c(v_2v_3) = \cdots = c(v_{k-1}v_k) = C_i\] and $c(v_kv_1) = C_j$.

Now, color the directed edges of $D$ ``Odd" or ``Even" depending on the parity of the number of vertices $k$ that gives the directed edge. Note that multiedges with different parities are possible in $D$.

\begin{figure}
        \begin{subfigure}[b]{0.18\textwidth}
          \centering
          \resizebox{\linewidth}{!}{
          \begin{tikzpicture}
\draw[very thick, black] (1.0, 0.0) -- (6.123233995736766e-17, 1.0);
\draw[very thick, black] (1.0, 0.0) -- (-1.0, 1.2246467991473532e-16);
\draw[very thick, green] (1.0, 0.0) -- (-1.8369701987210297e-16, -1.0);
\filldraw [black] (1.0, 0.0) circle (1pt);

\draw[very thick, red] (6.123233995736766e-17, 1.0) -- (-1.0, 1.2246467991473532e-16);
\draw[very thick, red] (6.123233995736766e-17, 1.0) -- (-1.8369701987210297e-16, -1.0);
\filldraw [black] (6.123233995736766e-17, 1.0) circle (1pt);

\draw[very thick, green] (-1.0, 1.2246467991473532e-16) -- (-1.8369701987210297e-16, -1.0);
\filldraw [black] (-1.0, 1.2246467991473532e-16) circle (1pt);

\filldraw [black] (-1.8369701987210297e-16, -1.0) circle (1pt);

\end{tikzpicture}
          }
     \end{subfigure}
     \begin{subfigure}[b]{0.18\textwidth}
          \centering
          \resizebox{\linewidth}{!}{
          \begin{tikzpicture}
\draw[very thick, black] (1.0, 0.0) -- (0.30901699437494745, 0.95105651629515353);
\draw[very thick, red] (1.0, 0.0) -- (-0.80901699437494734, 0.58778525229247325);
\filldraw [black] (1.0, 0.0) circle (1pt);

\draw[very thick, black] (0.30901699437494745, 0.95105651629515353) -- (-0.80901699437494734, 0.58778525229247325);
\filldraw [black] (0.30901699437494745, 0.95105651629515353) circle (1pt);

\draw[very thick, black] (-0.80901699437494734, 0.58778525229247325) -- (-0.80901699437494745, -0.58778525229247303);
\draw[very thick, red] (-0.80901699437494734, 0.58778525229247325) -- (0.30901699437494723, -0.95105651629515364);
\filldraw [black] (-0.80901699437494734, 0.58778525229247325) circle (1pt);

\draw[very thick, red] (-0.80901699437494745, -0.58778525229247303) -- (0.30901699437494723, -0.95105651629515364);
\filldraw [black] (-0.80901699437494745, -0.58778525229247303) circle (1pt);

\filldraw [black] (0.30901699437494723, -0.95105651629515364) circle (1pt);

\end{tikzpicture}
          }
     \end{subfigure}
      \begin{subfigure}[b]{0.18\textwidth}
          \centering
          \resizebox{\linewidth}{!}{
          \begin{tikzpicture}
\draw[very thick, black] (1.0, 0.0) -- (0.30901699437494745, 0.95105651629515353);
\draw[very thick, black] (1.0, 0.0) -- (-0.80901699437494745, -0.58778525229247303);

\draw[very thick, red] (-0.80901699437494734, 0.58778525229247325) -- (1.0, 0.0);
\draw[very thick, red] (-0.80901699437494734, 0.58778525229247325) -- (0.30901699437494723, -0.95105651629515364);

\draw[very thick, green] (0.30901699437494745, 0.95105651629515353) -- (-0.80901699437494734, 0.58778525229247325);
\draw[very thick, green] (0.30901699437494745, 0.95105651629515353) -- (-0.80901699437494745, -0.58778525229247303);

\draw[very thick, blue] (-0.80901699437494745, -0.58778525229247303) -- (-0.80901699437494734, 0.58778525229247325);
\draw[very thick, blue] (-0.80901699437494745, -0.58778525229247303) -- (0.30901699437494723, -0.95105651629515364);

\draw[very thick, yellow] (0.30901699437494723, -0.95105651629515364) -- (1.0, 0.0);
\draw[very thick, yellow] (0.30901699437494723, -0.95105651629515364) -- (0.30901699437494745, 0.95105651629515353);

\filldraw [black] (1.0, 0.0) circle (1pt);

\filldraw [black] (0.30901699437494745, 0.95105651629515353) circle (1pt);

\filldraw [black] (-0.80901699437494734, 0.58778525229247325) circle (1pt);

\filldraw [black] (-0.80901699437494745, -0.58778525229247303) circle (1pt);

\filldraw [black] (0.30901699437494723, -0.95105651629515364) circle (1pt);

\end{tikzpicture}
          }
     \end{subfigure}
     \begin{subfigure}[b]{0.18\textwidth}
          \centering
          \resizebox{\linewidth}{!}{
          \begin{tikzpicture}
\draw[very thick, black] (1.0, 0.0) -- (0.30901699437494745, 0.95105651629515353);

\draw[very thick, red] (-0.80901699437494734, 0.58778525229247325) -- (1.0, 0.0);
\draw[very thick, red] (-0.80901699437494734, 0.58778525229247325) -- (0.30901699437494723, -0.95105651629515364);

\draw[very thick, green] (0.30901699437494745, 0.95105651629515353) -- (-0.80901699437494745, -0.58778525229247303);

\draw[very thick, blue] (-0.80901699437494745, -0.58778525229247303) -- (-0.80901699437494734, 0.58778525229247325);
\draw[very thick, blue] (-0.80901699437494745, -0.58778525229247303) -- (0.30901699437494723, -0.95105651629515364);

\draw[very thick, black] (0.30901699437494723, -0.95105651629515364) -- (1.0, 0.0);
\draw[very thick, green] (0.30901699437494723, -0.95105651629515364) -- (0.30901699437494745, 0.95105651629515353);

\filldraw [black] (1.0, 0.0) circle (1pt);

\filldraw [black] (0.30901699437494745, 0.95105651629515353) circle (1pt);

\filldraw [black] (-0.80901699437494734, 0.58778525229247325) circle (1pt);

\filldraw [black] (-0.80901699437494745, -0.58778525229247303) circle (1pt);

\filldraw [black] (0.30901699437494723, -0.95105651629515364) circle (1pt);

\end{tikzpicture}
          }
     \end{subfigure}
     \caption{Four configurations that each contain a forbidden ``color cycle."}
     \label{colorcycles}
 \end{figure}
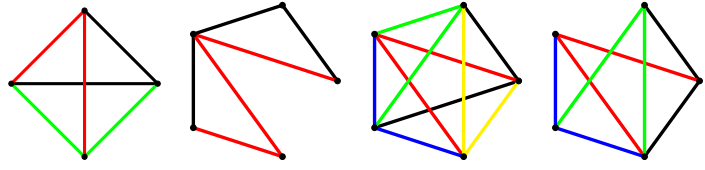

\begin{lemma}
\label{colorcycle}
The CFLS coloring avoids any edge-colored copy of $K_p$ for which the auxiliary digraph on the colors spanned by that clique contains a directed cycle with at least one Odd edge.
\end{lemma}

\begin{proof}
Suppose that colors $C_1,\ldots,C_m$ make such a directed cycle: \[C_1 \rightarrow C_2, C_2 \rightarrow C_3, \ldots, C_m \rightarrow C_1 \in E(D)\] such that (without loss of generality) $C_1 \rightarrow C_2$ is colored Odd. Then there exist vertices $v_1,\ldots,v_k \in V(K_p)$ for some odd integer $k \geq 3$ for which $\varphi_1(v_iv_{i+1}) = C_1$ for $i=1,\ldots,k-1$ and $\varphi_1(v_kv_1)=C_2$. Let the zero coordinate of the color $C_1$ be $(i,\{x,y\})$. Without loss, assume that $v_1^{(i)} = x$. It follows that $v_k^{(i)}=x$ as well. Therefore, the $i$th coordinate of $C_2$ is zero.

Now, each subsequent directed edge $C_j \rightarrow C_{j+1}$ for $j=2,\ldots,m$, regardless of color, forces the $i$th coordinate of the ``head" color to be zero as well. To see this assume that $C_j$ is zero in its $i$th coordinate. A monochromatic path in color $C_j$, $u_1 u_2 \cdots u_l$, implies that $u_1$ agrees with $u_l$ at $i$. Therefore, the $i$th coordinate of $C_{j+1}$ must also be zero. The same must be true for $C_1$ since this is a directed cycle, a contradiction.
\end{proof}

For the $(5,6)$-coloring we use Lemma~\ref{colorcycle} to eliminate the configurations shown in Figure~\ref{colorcycles} as well as monochromatic odd cycles.

\subsection{Configurations containing a monochromatic $P_3$}

Assume that we have an edge-colored $K_5$ that contains a monochromatic $P_3$ on vertices $abcd$ in color Black. The edges $ac$ and $bd$ cannot be Black since color classes are bipartite in CFLS. So we color edge $ac$ Red. Let $\text{Black}_0 = (i,\{x,y\})$, then $\text{Red}_i = 0$. Therefore, any Red edge from the Black $P_3$ to vertex $e$ fixes the value of $e^{(i)}$ as either $x$ or $y$. CFLS would then forbid any third color from having an edge between an $x$ and a $y$ as well as one between two vertices that agree at $i$.

There are 16 possible configurations that fit this description when we do not consider those in which the third color never touches vertex $e$ (these cases are summarized separately). Figure~\ref{TheP3s} presents these 16 configurations.

\begin{figure}

\begin{subfigure}[b]{0.18\textwidth}
          \centering
          \resizebox{\linewidth}{!}{
\begin{tikzpicture}
\draw[very thick, black] (1.0, 0.0) -- (0.30901699437494745, 0.95105651629515353);
\draw[very thick, red] (1.0, 0.0) -- (-0.80901699437494734, 0.58778525229247325);
\draw[very thick, red] (1.0, 0.0) -- (0.30901699437494723, -0.95105651629515364);
\filldraw [black] (1.0, 0.0) circle (1pt);
\draw[very thick, black] (0.30901699437494745, 0.95105651629515353) -- (-0.80901699437494734, 0.58778525229247325);
\draw[very thick, green] (0.30901699437494745, 0.95105651629515353) -- (-0.80901699437494745, -0.58778525229247303);
\filldraw [black] (0.30901699437494745, 0.95105651629515353) circle (1pt);
\draw[very thick, black] (-0.80901699437494734, 0.58778525229247325) -- (-0.80901699437494745, -0.58778525229247303);
\filldraw [black] (-0.80901699437494734, 0.58778525229247325) circle (1pt);
\draw[very thick, green] (-0.80901699437494745, -0.58778525229247303) -- (0.30901699437494723, -0.95105651629515364);
\filldraw [black] (-0.80901699437494745, -0.58778525229247303) circle (1pt);
\filldraw [black] (0.30901699437494723, -0.95105651629515364) circle (1pt);
\end{tikzpicture}
}

\end{subfigure}
\begin{subfigure}[b]{0.18\textwidth}
          \centering
          \resizebox{\linewidth}{!}{
\begin{tikzpicture}
\draw[very thick, black] (1.0, 0.0) -- (0.30901699437494745, 0.95105651629515353);
\draw[very thick, red] (1.0, 0.0) -- (-0.80901699437494734, 0.58778525229247325);
\draw[very thick, green] (1.0, 0.0) -- (-0.80901699437494745, -0.58778525229247303);
\draw[very thick, red] (1.0, 0.0) -- (0.30901699437494723, -0.95105651629515364);
\filldraw [black] (1.0, 0.0) circle (1pt);
\draw[very thick, black] (0.30901699437494745, 0.95105651629515353) -- (-0.80901699437494734, 0.58778525229247325);
\filldraw [black] (0.30901699437494745, 0.95105651629515353) circle (1pt);
\draw[very thick, black] (-0.80901699437494734, 0.58778525229247325) -- (-0.80901699437494745, -0.58778525229247303);
\draw[very thick, green] (-0.80901699437494734, 0.58778525229247325) -- (0.30901699437494723, -0.95105651629515364);
\filldraw [black] (-0.80901699437494734, 0.58778525229247325) circle (1pt);
\filldraw [black] (-0.80901699437494745, -0.58778525229247303) circle (1pt);
\filldraw [black] (0.30901699437494723, -0.95105651629515364) circle (1pt);
\end{tikzpicture}
}

\end{subfigure}
\begin{subfigure}[b]{0.18\textwidth}
          \centering
          \resizebox{\linewidth}{!}{
\begin{tikzpicture}
\draw[very thick, black] (1.0, 0.0) -- (0.30901699437494745, 0.95105651629515353);
\draw[very thick, red] (1.0, 0.0) -- (-0.80901699437494734, 0.58778525229247325);
\draw[very thick, red] (1.0, 0.0) -- (0.30901699437494723, -0.95105651629515364);
\filldraw [black] (1.0, 0.0) circle (1pt);
\draw[very thick, black] (0.30901699437494745, 0.95105651629515353) -- (-0.80901699437494734, 0.58778525229247325);
\draw[very thick, green] (0.30901699437494745, 0.95105651629515353) -- (0.30901699437494723, -0.95105651629515364);
\filldraw [black] (0.30901699437494745, 0.95105651629515353) circle (1pt);
\draw[very thick, black] (-0.80901699437494734, 0.58778525229247325) -- (-0.80901699437494745, -0.58778525229247303);
\draw[very thick, green] (-0.80901699437494734, 0.58778525229247325) -- (0.30901699437494723, -0.95105651629515364);
\filldraw [black] (-0.80901699437494734, 0.58778525229247325) circle (1pt);
\filldraw [black] (-0.80901699437494745, -0.58778525229247303) circle (1pt);
\filldraw [black] (0.30901699437494723, -0.95105651629515364) circle (1pt);
\end{tikzpicture}
}

\end{subfigure}
\begin{subfigure}[b]{0.18\textwidth}
          \centering
          \resizebox{\linewidth}{!}{
\begin{tikzpicture}
\draw[very thick, black] (1.0, 0.0) -- (0.30901699437494745, 0.95105651629515353);
\draw[very thick, red] (1.0, 0.0) -- (-0.80901699437494734, 0.58778525229247325);
\draw[very thick, red] (1.0, 0.0) -- (0.30901699437494723, -0.95105651629515364);
\filldraw [black] (1.0, 0.0) circle (1pt);
\draw[very thick, black] (0.30901699437494745, 0.95105651629515353) -- (-0.80901699437494734, 0.58778525229247325);
\filldraw [black] (0.30901699437494745, 0.95105651629515353) circle (1pt);
\draw[very thick, black] (-0.80901699437494734, 0.58778525229247325) -- (-0.80901699437494745, -0.58778525229247303);
\draw[very thick, green] (-0.80901699437494734, 0.58778525229247325) -- (0.30901699437494723, -0.95105651629515364);
\filldraw [black] (-0.80901699437494734, 0.58778525229247325) circle (1pt);
\draw[very thick, green] (-0.80901699437494745, -0.58778525229247303) -- (0.30901699437494723, -0.95105651629515364);
\filldraw [black] (-0.80901699437494745, -0.58778525229247303) circle (1pt);
\filldraw [black] (0.30901699437494723, -0.95105651629515364) circle (1pt);
\end{tikzpicture}
}

\end{subfigure}

\begin{subfigure}[b]{0.18\textwidth}
          \centering
          \resizebox{\linewidth}{!}{
\begin{tikzpicture}
\draw[very thick, black] (1.0, 0.0) -- (0.30901699437494745, 0.95105651629515353);
\draw[very thick, red] (1.0, 0.0) -- (-0.80901699437494734, 0.58778525229247325);
\filldraw [black] (1.0, 0.0) circle (1pt);
\draw[very thick, black] (0.30901699437494745, 0.95105651629515353) -- (-0.80901699437494734, 0.58778525229247325);
\draw[very thick, red] (0.30901699437494745, 0.95105651629515353) -- (0.30901699437494723, -0.95105651629515364);
\filldraw [black] (0.30901699437494745, 0.95105651629515353) circle (1pt);
\draw[very thick, black] (-0.80901699437494734, 0.58778525229247325) -- (-0.80901699437494745, -0.58778525229247303);
\draw[very thick, green] (-0.80901699437494734, 0.58778525229247325) -- (0.30901699437494723, -0.95105651629515364);
\filldraw [black] (-0.80901699437494734, 0.58778525229247325) circle (1pt);
\draw[very thick, green] (-0.80901699437494745, -0.58778525229247303) -- (0.30901699437494723, -0.95105651629515364);
\filldraw [black] (-0.80901699437494745, -0.58778525229247303) circle (1pt);
\filldraw [black] (0.30901699437494723, -0.95105651629515364) circle (1pt);
\end{tikzpicture}
}

\end{subfigure}
\begin{subfigure}[b]{0.18\textwidth}
          \centering
          \resizebox{\linewidth}{!}{
\begin{tikzpicture}
\draw[very thick, black] (1.0, 0.0) -- (0.30901699437494745, 0.95105651629515353);
\draw[very thick, red] (1.0, 0.0) -- (-0.80901699437494734, 0.58778525229247325);
\draw[very thick, green] (1.0, 0.0) -- (0.30901699437494723, -0.95105651629515364);
\filldraw [black] (1.0, 0.0) circle (1pt);
\draw[very thick, black] (0.30901699437494745, 0.95105651629515353) -- (-0.80901699437494734, 0.58778525229247325);
\draw[very thick, red] (0.30901699437494745, 0.95105651629515353) -- (0.30901699437494723, -0.95105651629515364);
\filldraw [black] (0.30901699437494745, 0.95105651629515353) circle (1pt);
\draw[very thick, black] (-0.80901699437494734, 0.58778525229247325) -- (-0.80901699437494745, -0.58778525229247303);
\filldraw [black] (-0.80901699437494734, 0.58778525229247325) circle (1pt);
\draw[very thick, green] (-0.80901699437494745, -0.58778525229247303) -- (0.30901699437494723, -0.95105651629515364);
\filldraw [black] (-0.80901699437494745, -0.58778525229247303) circle (1pt);
\filldraw [black] (0.30901699437494723, -0.95105651629515364) circle (1pt);
\end{tikzpicture}
}

\end{subfigure}
\begin{subfigure}[b]{0.18\textwidth}
          \centering
          \resizebox{\linewidth}{!}{
\begin{tikzpicture}
\draw[very thick, black] (1.0, 0.0) -- (0.30901699437494745, 0.95105651629515353);
\draw[very thick, red] (1.0, 0.0) -- (-0.80901699437494734, 0.58778525229247325);
\draw[very thick, green] (1.0, 0.0) -- (-0.80901699437494745, -0.58778525229247303);
\filldraw [black] (1.0, 0.0) circle (1pt);
\draw[very thick, black] (0.30901699437494745, 0.95105651629515353) -- (-0.80901699437494734, 0.58778525229247325);
\draw[very thick, red] (0.30901699437494745, 0.95105651629515353) -- (0.30901699437494723, -0.95105651629515364);
\filldraw [black] (0.30901699437494745, 0.95105651629515353) circle (1pt);
\draw[very thick, black] (-0.80901699437494734, 0.58778525229247325) -- (-0.80901699437494745, -0.58778525229247303);
\filldraw [black] (-0.80901699437494734, 0.58778525229247325) circle (1pt);
\draw[very thick, green] (-0.80901699437494745, -0.58778525229247303) -- (0.30901699437494723, -0.95105651629515364);
\filldraw [black] (-0.80901699437494745, -0.58778525229247303) circle (1pt);
\filldraw [black] (0.30901699437494723, -0.95105651629515364) circle (1pt);
\end{tikzpicture}
}

\end{subfigure}
\begin{subfigure}[b]{0.18\textwidth}
          \centering
          \resizebox{\linewidth}{!}{
\begin{tikzpicture}
\draw[very thick, black] (1.0, 0.0) -- (0.30901699437494745, 0.95105651629515353);
\draw[very thick, red] (1.0, 0.0) -- (-0.80901699437494734, 0.58778525229247325);
\filldraw [black] (1.0, 0.0) circle (1pt);
\draw[very thick, black] (0.30901699437494745, 0.95105651629515353) -- (-0.80901699437494734, 0.58778525229247325);
\draw[very thick, green] (0.30901699437494745, 0.95105651629515353) -- (-0.80901699437494745, -0.58778525229247303);
\draw[very thick, red] (0.30901699437494745, 0.95105651629515353) -- (0.30901699437494723, -0.95105651629515364);
\filldraw [black] (0.30901699437494745, 0.95105651629515353) circle (1pt);
\draw[very thick, black] (-0.80901699437494734, 0.58778525229247325) -- (-0.80901699437494745, -0.58778525229247303);
\draw[very thick, green] (-0.80901699437494734, 0.58778525229247325) -- (0.30901699437494723, -0.95105651629515364);
\filldraw [black] (-0.80901699437494734, 0.58778525229247325) circle (1pt);
\filldraw [black] (-0.80901699437494745, -0.58778525229247303) circle (1pt);
\filldraw [black] (0.30901699437494723, -0.95105651629515364) circle (1pt);
\end{tikzpicture}
}

\end{subfigure}

\begin{subfigure}[b]{0.18\textwidth}
          \centering
          \resizebox{\linewidth}{!}{
\begin{tikzpicture}
\draw[very thick, black] (1.0, 0.0) -- (0.30901699437494745, 0.95105651629515353);
\draw[very thick, red] (1.0, 0.0) -- (-0.80901699437494734, 0.58778525229247325);
\draw[very thick, green] (1.0, 0.0) -- (0.30901699437494723, -0.95105651629515364);
\filldraw [black] (1.0, 0.0) circle (1pt);
\draw[very thick, black] (0.30901699437494745, 0.95105651629515353) -- (-0.80901699437494734, 0.58778525229247325);
\draw[very thick, green] (0.30901699437494745, 0.95105651629515353) -- (-0.80901699437494745, -0.58778525229247303);
\draw[very thick, red] (0.30901699437494745, 0.95105651629515353) -- (0.30901699437494723, -0.95105651629515364);
\filldraw [black] (0.30901699437494745, 0.95105651629515353) circle (1pt);
\draw[very thick, black] (-0.80901699437494734, 0.58778525229247325) -- (-0.80901699437494745, -0.58778525229247303);
\filldraw [black] (-0.80901699437494734, 0.58778525229247325) circle (1pt);
\filldraw [black] (-0.80901699437494745, -0.58778525229247303) circle (1pt);
\filldraw [black] (0.30901699437494723, -0.95105651629515364) circle (1pt);
\end{tikzpicture}
}

\end{subfigure}
\begin{subfigure}[b]{0.18\textwidth}
          \centering
          \resizebox{\linewidth}{!}{
\begin{tikzpicture}
\draw[very thick, black] (1.0, 0.0) -- (0.30901699437494745, 0.95105651629515353);
\draw[very thick, red] (1.0, 0.0) -- (-0.80901699437494734, 0.58778525229247325);
\draw[very thick, green] (1.0, 0.0) -- (0.30901699437494723, -0.95105651629515364);
\filldraw [black] (1.0, 0.0) circle (1pt);
\draw[very thick, black] (0.30901699437494745, 0.95105651629515353) -- (-0.80901699437494734, 0.58778525229247325);
\draw[very thick, green] (0.30901699437494745, 0.95105651629515353) -- (0.30901699437494723, -0.95105651629515364);
\filldraw [black] (0.30901699437494745, 0.95105651629515353) circle (1pt);
\draw[very thick, black] (-0.80901699437494734, 0.58778525229247325) -- (-0.80901699437494745, -0.58778525229247303);
\draw[very thick, red] (-0.80901699437494734, 0.58778525229247325) -- (0.30901699437494723, -0.95105651629515364);
\filldraw [black] (-0.80901699437494734, 0.58778525229247325) circle (1pt);
\filldraw [black] (-0.80901699437494745, -0.58778525229247303) circle (1pt);
\filldraw [black] (0.30901699437494723, -0.95105651629515364) circle (1pt);
\end{tikzpicture}
}

\end{subfigure}
\begin{subfigure}[b]{0.18\textwidth}
          \centering
          \resizebox{\linewidth}{!}{
\begin{tikzpicture}
\draw[very thick, black] (1.0, 0.0) -- (0.30901699437494745, 0.95105651629515353);
\draw[very thick, red] (1.0, 0.0) -- (-0.80901699437494734, 0.58778525229247325);
\draw[very thick, green] (1.0, 0.0) -- (-0.80901699437494745, -0.58778525229247303);
\draw[very thick, green] (1.0, 0.0) -- (0.30901699437494723, -0.95105651629515364);
\filldraw [black] (1.0, 0.0) circle (1pt);
\draw[very thick, black] (0.30901699437494745, 0.95105651629515353) -- (-0.80901699437494734, 0.58778525229247325);
\filldraw [black] (0.30901699437494745, 0.95105651629515353) circle (1pt);
\draw[very thick, black] (-0.80901699437494734, 0.58778525229247325) -- (-0.80901699437494745, -0.58778525229247303);
\draw[very thick, red] (-0.80901699437494734, 0.58778525229247325) -- (0.30901699437494723, -0.95105651629515364);
\filldraw [black] (-0.80901699437494734, 0.58778525229247325) circle (1pt);
\filldraw [black] (-0.80901699437494745, -0.58778525229247303) circle (1pt);
\filldraw [black] (0.30901699437494723, -0.95105651629515364) circle (1pt);
\end{tikzpicture}
}

\end{subfigure}
\begin{subfigure}[b]{0.18\textwidth}
          \centering
          \resizebox{\linewidth}{!}{
\begin{tikzpicture}
\draw[very thick, black] (1.0, 0.0) -- (0.30901699437494745, 0.95105651629515353);
\draw[very thick, red] (1.0, 0.0) -- (-0.80901699437494734, 0.58778525229247325);
\draw[very thick, green] (1.0, 0.0) -- (0.30901699437494723, -0.95105651629515364);
\filldraw [black] (1.0, 0.0) circle (1pt);
\draw[very thick, black] (0.30901699437494745, 0.95105651629515353) -- (-0.80901699437494734, 0.58778525229247325);
\filldraw [black] (0.30901699437494745, 0.95105651629515353) circle (1pt);
\draw[very thick, black] (-0.80901699437494734, 0.58778525229247325) -- (-0.80901699437494745, -0.58778525229247303);
\draw[very thick, red] (-0.80901699437494734, 0.58778525229247325) -- (0.30901699437494723, -0.95105651629515364);
\filldraw [black] (-0.80901699437494734, 0.58778525229247325) circle (1pt);
\draw[very thick, green] (-0.80901699437494745, -0.58778525229247303) -- (0.30901699437494723, -0.95105651629515364);
\filldraw [black] (-0.80901699437494745, -0.58778525229247303) circle (1pt);
\filldraw [black] (0.30901699437494723, -0.95105651629515364) circle (1pt);
\end{tikzpicture}
}

\end{subfigure}

\begin{subfigure}[b]{0.18\textwidth}
          \centering
          \resizebox{\linewidth}{!}{
\begin{tikzpicture}
\draw[very thick, black] (1.0, 0.0) -- (0.30901699437494745, 0.95105651629515353);
\draw[very thick, red] (1.0, 0.0) -- (-0.80901699437494734, 0.58778525229247325);
\filldraw [black] (1.0, 0.0) circle (1pt);
\draw[very thick, black] (0.30901699437494745, 0.95105651629515353) -- (-0.80901699437494734, 0.58778525229247325);
\draw[very thick, green] (0.30901699437494745, 0.95105651629515353) -- (-0.80901699437494745, -0.58778525229247303);
\draw[very thick, green] (0.30901699437494745, 0.95105651629515353) -- (0.30901699437494723, -0.95105651629515364);
\filldraw [black] (0.30901699437494745, 0.95105651629515353) circle (1pt);
\draw[very thick, black] (-0.80901699437494734, 0.58778525229247325) -- (-0.80901699437494745, -0.58778525229247303);
\draw[very thick, red] (-0.80901699437494734, 0.58778525229247325) -- (0.30901699437494723, -0.95105651629515364);
\filldraw [black] (-0.80901699437494734, 0.58778525229247325) circle (1pt);
\filldraw [black] (-0.80901699437494745, -0.58778525229247303) circle (1pt);
\filldraw [black] (0.30901699437494723, -0.95105651629515364) circle (1pt);
\end{tikzpicture}
}

\end{subfigure}
\begin{subfigure}[b]{0.18\textwidth}
          \centering
          \resizebox{\linewidth}{!}{
\begin{tikzpicture}
\draw[very thick, black] (1.0, 0.0) -- (0.30901699437494745, 0.95105651629515353);
\draw[very thick, red] (1.0, 0.0) -- (-0.80901699437494734, 0.58778525229247325);
\filldraw [black] (1.0, 0.0) circle (1pt);
\draw[very thick, black] (0.30901699437494745, 0.95105651629515353) -- (-0.80901699437494734, 0.58778525229247325);
\draw[very thick, green] (0.30901699437494745, 0.95105651629515353) -- (-0.80901699437494745, -0.58778525229247303);
\filldraw [black] (0.30901699437494745, 0.95105651629515353) circle (1pt);
\draw[very thick, black] (-0.80901699437494734, 0.58778525229247325) -- (-0.80901699437494745, -0.58778525229247303);
\draw[very thick, red] (-0.80901699437494734, 0.58778525229247325) -- (0.30901699437494723, -0.95105651629515364);
\filldraw [black] (-0.80901699437494734, 0.58778525229247325) circle (1pt);
\draw[very thick, green] (-0.80901699437494745, -0.58778525229247303) -- (0.30901699437494723, -0.95105651629515364);
\filldraw [black] (-0.80901699437494745, -0.58778525229247303) circle (1pt);
\filldraw [black] (0.30901699437494723, -0.95105651629515364) circle (1pt);
\end{tikzpicture}
}

\end{subfigure}
\begin{subfigure}[b]{0.18\textwidth}
          \centering
          \resizebox{\linewidth}{!}{
\begin{tikzpicture}
\draw[very thick, black] (1.0, 0.0) -- (0.30901699437494745, 0.95105651629515353);
\draw[very thick, red] (1.0, 0.0) -- (-0.80901699437494734, 0.58778525229247325);
\draw[very thick, green] (1.0, 0.0) -- (0.30901699437494723, -0.95105651629515364);
\filldraw [black] (1.0, 0.0) circle (1pt);
\draw[very thick, black] (0.30901699437494745, 0.95105651629515353) -- (-0.80901699437494734, 0.58778525229247325);
\draw[very thick, green] (0.30901699437494745, 0.95105651629515353) -- (-0.80901699437494745, -0.58778525229247303);
\filldraw [black] (0.30901699437494745, 0.95105651629515353) circle (1pt);
\draw[very thick, black] (-0.80901699437494734, 0.58778525229247325) -- (-0.80901699437494745, -0.58778525229247303);
\filldraw [black] (-0.80901699437494734, 0.58778525229247325) circle (1pt);
\draw[very thick, red] (-0.80901699437494745, -0.58778525229247303) -- (0.30901699437494723, -0.95105651629515364);
\filldraw [black] (-0.80901699437494745, -0.58778525229247303) circle (1pt);
\filldraw [black] (0.30901699437494723, -0.95105651629515364) circle (1pt);
\end{tikzpicture}
}

\end{subfigure}
\begin{subfigure}[b]{0.18\textwidth}
          \centering
          \resizebox{\linewidth}{!}{
\begin{tikzpicture}
\draw[very thick, black] (1.0, 0.0) -- (0.30901699437494745, 0.95105651629515353);
\draw[very thick, red] (1.0, 0.0) -- (-0.80901699437494734, 0.58778525229247325);
\draw[very thick, green] (1.0, 0.0) -- (-0.80901699437494745, -0.58778525229247303);
\filldraw [black] (1.0, 0.0) circle (1pt);
\draw[very thick, black] (0.30901699437494745, 0.95105651629515353) -- (-0.80901699437494734, 0.58778525229247325);
\draw[very thick, green] (0.30901699437494745, 0.95105651629515353) -- (0.30901699437494723, -0.95105651629515364);
\filldraw [black] (0.30901699437494745, 0.95105651629515353) circle (1pt);
\draw[very thick, black] (-0.80901699437494734, 0.58778525229247325) -- (-0.80901699437494745, -0.58778525229247303);
\filldraw [black] (-0.80901699437494734, 0.58778525229247325) circle (1pt);
\draw[very thick, red] (-0.80901699437494745, -0.58778525229247303) -- (0.30901699437494723, -0.95105651629515364);
\filldraw [black] (-0.80901699437494745, -0.58778525229247303) circle (1pt);
\filldraw [black] (0.30901699437494723, -0.95105651629515364) circle (1pt);
\end{tikzpicture}
}

\end{subfigure}
\caption{CFLS forbids these 16 configurations, each contain a monochromatic $P_3$.}
\label{TheP3s}

\end{figure}
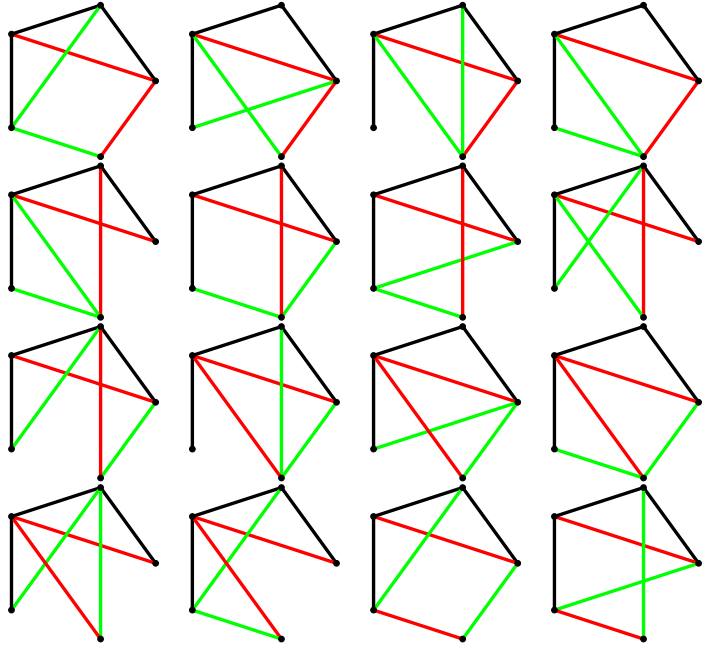

\subsection{Configurations containing an alternating $C_4$}

Next, consider configurations that contain a 2-colored $C_4$ with alternating colors. If the configuration also has two same-colored edges adjacent at the fifth vertex so that the other two endpoints are each incident to either endpoint of an edge of the $C_4$ (as shown by the edge-colored cliques in Figure~\ref{altC4s}), then we can say that under CFLS, the color from the fifth vertex must be distinct from any color spanned by the other four vertices. Moreover, none of these spanned colors can be incident with the fifth vertex.

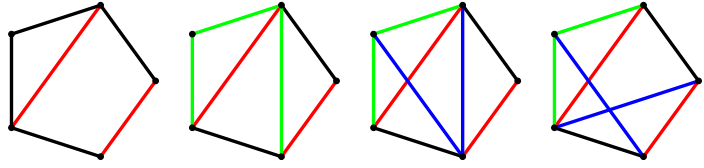
\begin{figure}
\begin{subfigure}[b]{0.18\textwidth}
          \centering
          \resizebox{\linewidth}{!}{
\begin{tikzpicture}
\draw[very thick, red] (1.0, 0.0) -- (0.30901699437494723, -0.95105651629515364);
\draw[very thick, black] (1.0, 0.0) -- (0.30901699437494745, 0.95105651629515353);
\draw[very thick, red] (0.30901699437494745, 0.95105651629515353) -- (-0.80901699437494745, -0.58778525229247303);
\draw[very thick, black] (-0.80901699437494745, -0.58778525229247303) -- (0.30901699437494723, -0.95105651629515364);

\draw[very thick, black] (-0.80901699437494734, 0.58778525229247325) -- (-0.80901699437494745, -0.58778525229247303);
\draw[very thick, black] (-0.80901699437494734, 0.58778525229247325) -- (0.30901699437494745, 0.95105651629515353);

\filldraw [black] (1.0, 0.0) circle (1pt);

\filldraw [black] (0.30901699437494745, 0.95105651629515353) circle (1pt);

\filldraw [black] (-0.80901699437494734, 0.58778525229247325) circle (1pt);

\filldraw [black] (-0.80901699437494745, -0.58778525229247303) circle (1pt);

\filldraw [black] (0.30901699437494723, -0.95105651629515364) circle (1pt);

\end{tikzpicture}
}
\end{subfigure}
\begin{subfigure}[b]{0.18\textwidth}
          \centering
          \resizebox{\linewidth}{!}{
\begin{tikzpicture}
\draw[very thick, red] (1.0, 0.0) -- (0.30901699437494723, -0.95105651629515364);
\draw[very thick, black] (1.0, 0.0) -- (0.30901699437494745, 0.95105651629515353);
\draw[very thick, red] (0.30901699437494745, 0.95105651629515353) -- (-0.80901699437494745, -0.58778525229247303);
\draw[very thick, black] (-0.80901699437494745, -0.58778525229247303) -- (0.30901699437494723, -0.95105651629515364);

\draw[very thick, green] (-0.80901699437494734, 0.58778525229247325) -- (-0.80901699437494745, -0.58778525229247303);
\draw[very thick, green] (-0.80901699437494734, 0.58778525229247325) -- (0.30901699437494745, 0.95105651629515353);
\draw[very thick, green] (0.30901699437494745, 0.95105651629515353) -- (0.30901699437494723, -0.95105651629515364);

\filldraw [black] (1.0, 0.0) circle (1pt);

\filldraw [black] (0.30901699437494745, 0.95105651629515353) circle (1pt);

\filldraw [black] (-0.80901699437494734, 0.58778525229247325) circle (1pt);

\filldraw [black] (-0.80901699437494745, -0.58778525229247303) circle (1pt);

\filldraw [black] (0.30901699437494723, -0.95105651629515364) circle (1pt);

\end{tikzpicture}
}
\end{subfigure}
\begin{subfigure}[b]{0.18\textwidth}
          \centering
          \resizebox{\linewidth}{!}{
\begin{tikzpicture}
\draw[very thick, red] (1.0, 0.0) -- (0.30901699437494723, -0.95105651629515364);
\draw[very thick, black] (1.0, 0.0) -- (0.30901699437494745, 0.95105651629515353);
\draw[very thick, red] (0.30901699437494745, 0.95105651629515353) -- (-0.80901699437494745, -0.58778525229247303);
\draw[very thick, black] (-0.80901699437494745, -0.58778525229247303) -- (0.30901699437494723, -0.95105651629515364);

\draw[very thick, green] (-0.80901699437494734, 0.58778525229247325) -- (-0.80901699437494745, -0.58778525229247303);
\draw[very thick, green] (-0.80901699437494734, 0.58778525229247325) -- (0.30901699437494745, 0.95105651629515353);

\draw[very thick, blue] (-0.80901699437494734, 0.58778525229247325) -- (0.30901699437494723, -0.95105651629515364);
\draw[very thick, blue] (0.30901699437494723, -0.95105651629515364) -- (0.30901699437494745, 0.95105651629515353);

\filldraw [black] (1.0, 0.0) circle (1pt);

\filldraw [black] (0.30901699437494745, 0.95105651629515353) circle (1pt);

\filldraw [black] (-0.80901699437494734, 0.58778525229247325) circle (1pt);

\filldraw [black] (-0.80901699437494745, -0.58778525229247303) circle (1pt);

\filldraw [black] (0.30901699437494723, -0.95105651629515364) circle (1pt);

\end{tikzpicture}
}
\end{subfigure}
\begin{subfigure}[b]{0.18\textwidth}
          \centering
          \resizebox{\linewidth}{!}{
\begin{tikzpicture}
\draw[very thick, red] (1.0, 0.0) -- (0.30901699437494723, -0.95105651629515364);
\draw[very thick, black] (1.0, 0.0) -- (0.30901699437494745, 0.95105651629515353);
\draw[very thick, red] (0.30901699437494745, 0.95105651629515353) -- (-0.80901699437494745, -0.58778525229247303);
\draw[very thick, black] (-0.80901699437494745, -0.58778525229247303) -- (0.30901699437494723, -0.95105651629515364);

\draw[very thick, green] (-0.80901699437494734, 0.58778525229247325) -- (-0.80901699437494745, -0.58778525229247303);
\draw[very thick, green] (-0.80901699437494734, 0.58778525229247325) -- (0.30901699437494745, 0.95105651629515353);

\draw[very thick, blue] (-0.80901699437494734, 0.58778525229247325) -- (0.30901699437494723, -0.95105651629515364);
\draw[very thick, blue] (1,0) -- (-0.80901699437494745, -0.58778525229247303);

\filldraw [black] (1.0, 0.0) circle (1pt);

\filldraw [black] (0.30901699437494745, 0.95105651629515353) circle (1pt);

\filldraw [black] (-0.80901699437494734, 0.58778525229247325) circle (1pt);

\filldraw [black] (-0.80901699437494745, -0.58778525229247303) circle (1pt);

\filldraw [black] (0.30901699437494723, -0.95105651629515364) circle (1pt);

\end{tikzpicture}
}
\end{subfigure}

\caption{Configurations containing an alternating $C_4$.}
\label{altC4s}
\end{figure}

\begin{lemma}
Let $a,b,c,d,e$ be distinct vertices such that $\varphi(ab)=\varphi(cd) = \alpha$, $\varphi(bc) = \varphi(ad) = \beta$, $\varphi(ae)=\varphi(de) = \gamma$, $\varphi(ac)=\pi_1$, $\varphi(bd) = \pi_2$, $\varphi(be) = \lambda_1$, and $\varphi(ce) = \lambda_2$. Then \[\gamma, \lambda_1, \lambda_2 \not \in \{\alpha, \beta, \pi_1, \pi_2\}.\]
\end{lemma}

\begin{proof}
Let $\beta_0 = (i,\{x,y\})$. Without loss of generality, we may assume that $a^{(i)} = x$ and $d^{(i)}=y$. If $e^{(i)} = x$ or $e^{(i)} = y$, then we get that $\gamma_i = 0$ and $\gamma_i \neq 0$, a contradiction. Hence, \[e^{(i)} = z \not \in \{x,y\}.\] This alone shows that $\gamma \neq \beta$.

If $\gamma = \alpha$, then $\alpha_i \neq 0$ so it must be the case that $b^{(i)}=y$ and $c^{(i)}=x$. Therefore, three distinct binary strings, $x,y,z$, pairwise have the same first index of difference $\alpha_i$. This is impossible since two must either both be zero or both be one.

If $\gamma = \pi_1$, then $c^{(i)}=y$ and so $b^{(i)}=x$. Hence, the distinct binary strings $x,y,z$ again pairwise have the same first index of difference, $\gamma_i$, a contradiction. The same argument applies if $\gamma = \pi_2$.

Next, assume that $\lambda_1=\pi_1$. Since $e^{(i)}=z$ and $b^{(i)} \in \{x,y\}$, then $\lambda_1$ is nonzero at $i$. Since $\varphi(ac) = \lambda_1$ and $a^{(i)} = x$, then $c^{(i)} = y$ and so $b^{(i)}=x$. Hence, colors $\gamma$ and $\lambda_1$ must agree in coordinate $i$. This again gives us that $x,y,z$ all pairwise differ at the same first index, a contradiction. The same argument applies if $\lambda_2=\pi_2$.

If $\lambda_1=\pi_2$, then $b^{(i)}=x$ and so $c^{(i)}=y$. So again $\gamma$ and $\lambda_1$ agree at coordinate $i$. This again forces the contradiction with $x,y,z$. The same argument applies if $\lambda_2=\pi_1$.

Finally, note that $\lambda_1,\lambda_2 \neq \beta$ since $e^{(i)}=z$. If $\lambda_1 = \alpha$, then $b^{(i)}=y$ and so $\gamma$ and $\alpha$ agree at coordinate $i$ which gives the same contradiction as before. The same reasoning applies if $\lambda_2=\alpha$.
\end{proof}

\subsection{Additional configurations}

We finish this section by showing that the modified CFLS coloring eliminates the additional cases shown in Figure~\ref{additional}. We already showed in \cite{55} that this coloring eliminates Figures~\ref{fig:A} and~\ref{fig:E}, but we reproduce the proofs here for completeness.

\begin{lemma}
\label{forbidden}
The CFLS coloring forbids four distinct vertices $a,b,c,d \in V$ for which $\varphi_1(a,b)=\varphi_1(c,d)$ and $\varphi_1(a,c)=\varphi_1(a,d)$ (see Figure~\ref{fig:A}).
\end{lemma}

\begin{proof}
Assume towards a contradiction that $\varphi_1(a,b)=\varphi_1(c,d)=\alpha$ and $\varphi_1(a,c)=\varphi_1(a,d)=\gamma$. Let $\alpha_0=(i,\{x,y\})$. Without loss of generality, $a^{(i)}=c^{(i)}=x$ and $b^{(i)}=d^{(i)}=y$. Then $\gamma_i = 0$ since $a$ and $c$ agree at $i$, but $\gamma_i \neq 0$ as $a$ and $d$ differ at $i$, a contradiction.
\end{proof}

\begin{lemma}
The modified CFLS coloring $\varphi$ forbids four distinct vertices $a,b,c,d \in V$ with $\varphi(a,b)=\varphi(c,d)$, $\varphi(a,c)=\varphi(b,d)$, and $\varphi(a,d)=\varphi(b,c)$ (see Figure~\ref{fig:E}).
\end{lemma}
    
\begin{proof}
Assume towards a contradiction that a striped $K_4$ can occur. Then, $\varphi_1(a,b)=\varphi_1(c,d) = \alpha$, $\varphi_1(a,c)=\varphi_1(b,d) = \gamma$, and $\varphi_1(a,d)=\varphi_1(b,c)=\pi$. Let $\alpha_0 = (i,\{x,y\})$, $\gamma_0=(j,\{s,t\})$, and $\pi_0=(k,\{v,w\})$. Without loss of generality, assume that $i = \min\{i,j,k\}$. Since exactly one of $d^{(i)}$ and $c^{(i)}$ equals $a^{(i)}$, then either $j=i$ or $k=i$. Moreover, the other index must be strictly greater than $i$. So we may assume that $j=i$ and that $i < k$.

Let $a^{(i)}=d^{(i)}=x$, $b^{(i)}=c^{(i)}=y$, $a^{(k)}=b^{(k)}=v$, and $c^{(k)}=d^{(k)}=w$. Without loss of generality we may assume that $x<y$. This implies that $a,d < b,c$. If $v < w$, then $\delta_k(a,c) = +1$ and $\delta_k(d,b) = -1$. Therefore, $\varphi_2(a,c) \neq \varphi_2(b,d)$, a contradiction. So, it must be the case that $w < v$. But then $\delta_k(a,c)=-1$ and $\delta_k(b,d)=+1$, which yields the same contradiction.
\end{proof}

\begin{figure}
     \begin{subfigure}[b]{0.18\textwidth}
          \centering
          \resizebox{\linewidth}{!}{
          \begin{tikzpicture}
		\draw[very thick, black] (1.0, 0.0) -- (6.123233995736766e-17, 1.0);
		\draw[very thick, red] (1.0, 0.0) -- (-1.0, 1.2246467991473532e-16);
		\draw[very thick, red] (1.0, 0.0) -- (-1.8369701987210297e-16, -1.0);
		\filldraw [black] (1.0, 0.0) circle (1pt);
	
		\filldraw [black] (6.123233995736766e-17, 1.0) circle (1pt);
	
		\draw[very thick, black] (-1.0, 1.2246467991473532e-16) -- (-1.8369701987210297e-16, -1.0);
		\filldraw [black] (-1.0, 1.2246467991473532e-16) circle (1pt);
		
		\filldraw [black] (-1.8369701987210297e-16, -1.0) circle (1pt);
		
		\node[right] at (1,0) {$a$};
		\node[above] at (0,1) {$b$};
		\node[left] at (-1,0) {$c$};
		\node[below] at (0,-1) {$d$};
		
	\end{tikzpicture}
					}
          \caption{}
          \label{fig:A}
     \end{subfigure}
     \begin{subfigure}[b]{0.18\textwidth}
     \centering
     \resizebox{\linewidth}{!}{
     \begin{tikzpicture}
\draw[very thick, black] (1.0, 0.0) -- (6.123233995736766e-17, 1.0);
\draw[very thick, red] (1.0, 0.0) -- (-1.0, 1.2246467991473532e-16);
\draw[very thick, green] (1.0, 0.0) -- (-1.8369701987210297e-16, -1.0);
\filldraw [black] (1.0, 0.0) circle (1pt);
\draw[very thick, green] (6.123233995736766e-17, 1.0) -- (-1.0, 1.2246467991473532e-16);
\draw[very thick, red] (6.123233995736766e-17, 1.0) -- (-1.8369701987210297e-16, -1.0);
\filldraw [black] (6.123233995736766e-17, 1.0) circle (1pt);
\draw[very thick, black] (-1.0, 1.2246467991473532e-16) -- (-1.8369701987210297e-16, -1.0);
\filldraw [black] (-1.0, 1.2246467991473532e-16) circle (1pt);
\filldraw [black] (-1.8369701987210297e-16, -1.0) circle (1pt);

		\node[right] at (1,0) {$a$};
		\node[above] at (0,1) {$b$};
		\node[left] at (-1,0) {$c$};
		\node[below] at (0,-1) {$d$};
\end{tikzpicture}
}
\caption{}
\label{fig:E}
     \end{subfigure}
     \begin{subfigure}[b]{0.2\textwidth}
          \centering
          \resizebox{\linewidth}{!}{
 \begin{tikzpicture}
\draw[very thick, black] (1.0, 0.0) -- (0.30901699437494745, 0.95105651629515353);
\draw[very thick, black] (0.30901699437494745, 0.95105651629515353) -- (-0.80901699437494734, 0.58778525229247325);

\draw[very thick, red] (-0.80901699437494734, 0.58778525229247325) -- (-0.80901699437494745, -0.58778525229247303);
\draw[very thick, black] (-0.80901699437494745, -0.58778525229247303) --  (0.30901699437494723, -0.95105651629515364);

\draw[very thick, red] (1.0, 0.0) --  (0.30901699437494723, -0.95105651629515364);

\filldraw [black] (1.0, 0.0) circle (1pt);
\node[right] at (1,0) {$a$};
\filldraw [black] (0.30901699437494745, 0.95105651629515353) circle (1pt);
\node[above] at (0.30901699437494745, 0.95105651629515353) {$b$};
\filldraw [black] (-0.80901699437494734, 0.58778525229247325) circle (1pt);
\node[left] at (-0.80901699437494734, 0.58778525229247325) {$c$};
\filldraw [black] (-0.80901699437494745, -0.58778525229247303) circle (1pt);
\node[left] at (-0.80901699437494745, -0.58778525229247303) {$d$};
\filldraw [black] (0.30901699437494723, -0.95105651629515364) circle (1pt);
\node[below] at (0.30901699437494723, -0.95105651629515364) {$e$};
\end{tikzpicture}
					}
          \caption{}
          \label{fig:AA}
     \end{subfigure}
     \begin{subfigure}[b]{0.2\textwidth}
          \centering
          \resizebox{\linewidth}{!}{
  \begin{tikzpicture}
\draw[very thick, black] (1.0, 0.0) -- (0.30901699437494745, 0.95105651629515353);
\draw[very thick, red] (0.30901699437494745, 0.95105651629515353) -- (-0.80901699437494734, 0.58778525229247325);

\draw[very thick, black] (-0.80901699437494734, 0.58778525229247325) -- (-0.80901699437494745, -0.58778525229247303);
\draw[very thick, blue] (-0.80901699437494745, -0.58778525229247303) --  (0.30901699437494723, -0.95105651629515364);

\draw[very thick, red] (1.0, 0.0) --  (0.30901699437494723, -0.95105651629515364);

\draw[very thick, blue] (1.0, 0.0) --  (-0.80901699437494734, 0.58778525229247325);

\filldraw [black] (1.0, 0.0) circle (1pt);
\node[right] at (1,0) {$a$};
\filldraw [black] (0.30901699437494745, 0.95105651629515353) circle (1pt);
\node[above] at (0.30901699437494745, 0.95105651629515353) {$b$};
\filldraw [black] (-0.80901699437494734, 0.58778525229247325) circle (1pt);
\node[left] at (-0.80901699437494734, 0.58778525229247325) {$c$};
\filldraw [black] (-0.80901699437494745, -0.58778525229247303) circle (1pt);
\node[left] at (-0.80901699437494745, -0.58778525229247303) {$d$};
\filldraw [black] (0.30901699437494723, -0.95105651629515364) circle (1pt);
\node[below] at (0.30901699437494723, -0.95105651629515364) {$e$};
\end{tikzpicture}
          }
          \caption{}
          \label{fig:BB}
     \end{subfigure}
     \begin{subfigure}[b]{0.2\textwidth}
          \centering
          \resizebox{\linewidth}{!}{
\begin{tikzpicture}
\draw[very thick, red] (1.0, 0.0) -- (0.30901699437494723, -0.95105651629515364);
\draw[very thick, black] (1.0, 0.0) -- (0.30901699437494745, 0.95105651629515353);
\draw[very thick, red] (0.30901699437494745, 0.95105651629515353) -- (-0.80901699437494745, -0.58778525229247303);
\draw[very thick, black] (-0.80901699437494745, -0.58778525229247303) -- (0.30901699437494723, -0.95105651629515364);

\draw[very thick, blue] (0.30901699437494745, 0.95105651629515353) -- (0.30901699437494723, -0.95105651629515364);
\draw[very thick, blue] (0.30901699437494745, 0.95105651629515353) -- (-0.80901699437494734, 0.58778525229247325);
\draw[very thick, green] (0.30901699437494723, -0.95105651629515364) -- (-0.80901699437494734, 0.58778525229247325);
\draw[very thick, green] (-0.80901699437494745, -0.58778525229247303) -- (1,0);

\filldraw [black] (1.0, 0.0) circle (1pt);
\node[right] at (1,0) {$b$};
\filldraw [black] (0.30901699437494745, 0.95105651629515353) circle (1pt);
\node[above] at (0.30901699437494723, 0.95105651629515364) {$a$};
\filldraw [black] (-0.80901699437494734, 0.58778525229247325) circle (1pt);
\node[left] at (-0.80901699437494734, 0.58778525229247325) {$e$};
\filldraw [black] (-0.80901699437494745, -0.58778525229247303) circle (1pt);
\node[left] at (-0.80901699437494734, -0.58778525229247325) {$d$};
\filldraw [black] (0.30901699437494723, -0.95105651629515364) circle (1pt);
\node[below] at (0.30901699437494723, -0.95105651629515364) {$c$};
\end{tikzpicture}
}
\caption{}
\label{fig:altC4strag}
\end{subfigure}
     \caption{Five additional configurations avoided by $\varphi$.}
     \label{additional}
 \end{figure}
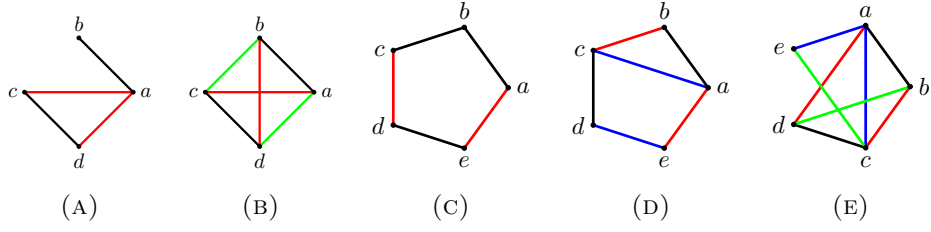

\begin{lemma}
Let $a,b,c,d,e$ be distinct vertices. The CFLS coloring forbids $\varphi(ab)=\varphi(bc)=\varphi(de)=\alpha$ and $\varphi(cd)=\varphi(ae)=\beta$ (see Figure~\ref{fig:AA}).
\end{lemma}

\begin{proof}
Assume that this can happen and that $\alpha_0=(i,\{x,y\})$. Without loss of generality, assume that $b^{(i)}=x$ and $a^{(i)}=c^{(i)}=y$. Moreover, without loss we can assume that $d^{(i)}=x$ and $e^{(i)}=y$. Then $a^{(i)}=e^{(i)}$ implies that $\beta_i=0$. But $c^{(i)} \neq d^{(i)}$ implies that $\beta_i=0$, a contradiction.
\end{proof}

\begin{lemma}
Let $a,b,c,d,e$ be distinct vertices. The CFLS coloring forbids $\varphi(ab)=\varphi(cd)=\alpha$, $\varphi(ae)=\varphi(bc)=\beta$, and $\varphi(ac)=\varphi(de)=\gamma$ (see Figure~\ref{fig:BB}).
\end{lemma}

\begin{proof}
Let $\gamma_0=(i,\{x,y\})$. Without loss of generality we may assume that $a^{(i)}=x$ and $c^{(i)}=y$. If $d^{(i)}=y$ and $e^{(i)}=x$, then $\alpha_i=\beta_i=0$ and so $b^{(i)}=x$ and $b^{(i)}=y$, a contradiction. Hence, $d^{(i)}=x$ and $e^{(i)}=y$. So $b^{(i)} = z \not \in \{x,y\}$. Moreover, $\alpha_i=\beta_i$ both go between $x$ and $y$. Therefore, the three distinct binary strings $x,y,z$ are all pairwise different at the same first index, $\alpha_i$, a contradiction.
\end{proof}

\begin{lemma}
It cannot be the case that $a,b,c,d,e$ are distinct vertices such that $\varphi(ab)=\varphi(cd) = \alpha$, $\varphi(bc) = \varphi(ad) = \beta$, $\varphi(ae)=\varphi(ac) = \gamma$, and $\varphi(bd)=\varphi(bc)=\pi$ (see Figure~\ref{fig:altC4strag}).
\end{lemma}

\begin{proof}
Let $\pi_0 = (i,\{x,y\})$. Without loss of generality we may assume that $e^{(i)}=x$ and $c^{(i)}=y$. It follows from color $\gamma$ that $a^{(i)} \not \in \{x,y\}$. Since $d^{(i)} \in \{x,y\}$, then it follows that $\beta_i \neq 0$. Hence, $b^{(i)} = x$ and so $d^{(i)}=y$. So $\varphi(cd)=\alpha$ implies that $\alpha_i=0$, but $\varphi(ab) = \alpha$ implies that $\alpha_i \neq 0$, a contradiction.
\end{proof}

Finally, another brief but important fact about the CFLS when considered along with a linear order on the binary strings that we used in \cite{55} and will use here in combination with the algebraic coloring is given in the following lemma.
 
 \begin{lemma}
\label{order}
If $a<b<c$, then $\varphi(a,b) \neq \varphi(b,c)$.
\end{lemma}

\begin{proof}
Suppose $\varphi_1(a,b)=\varphi_1(b,c) = \alpha$ and that $\alpha_0 = (i, \{x,y\})$ for $x < y$. Then $a^{(i)} = x$ and $b^{(i)}=y$. But then $c^{(i)}=x$. Therefore, $c < b$, a contradiction.
\end{proof}

 \section{Configurations avoided by the algebraic coloring}
 
 We begin with two basic lemmas about the algebraic construction $\chi$.
 
 \begin{lemma}
 \label{basic1}
 Let $a,b,c$ be three distinct vertices such that $\chi(ab)=\chi(ac)$, then $b_1 \neq c_1$.
 \end{lemma}
 
 \begin{proof}
 Let  $\chi(ab)=\chi(ac)$. Then
 \begin{align*}
 a_1b_1 - a_2 - b_2 &= a_1c_1 - a_2 - c_2\\
 a_1(b_1-c_1) &= b_2-c_2.
 \end{align*}
 If $b_1=c_1$, then $b_2=c_2$ as well. Hence, $b=c$, a contradiction.
 \end{proof}
 
 \begin{lemma}
 \label{basic2}
  Let $a,b,c,d$ be four distinct vertices such that $\chi(ab)=\chi(ac)$ and $\chi(db)=\chi(dc)$, then $a_1 = d_1$.
 \end{lemma}
 
 \begin{proof}
 Since $b_1 \neq c_1$ by Lemma~\ref{basic1}, then we know that \[a_1 = \frac{b_2-c_2}{b_1-c_1} = d_1.\]
 \end{proof}
 
 As shown in \cite{mubayi2004}, the algebraic construction $\chi_1$ avoids monochromatic $C_4$s (see Figure~\ref{fig:DM1}) as well as the configuration shown in Figure~\ref{fig:DM2}. We provide these two results for completeness.
 
 \begin{lemma}
 Let $a,b,c,d$ be distinct vertices. The algebraic coloring $\chi$ forbids $\chi(ab)=\chi(bc)=\chi(cd)=\chi(da)$ (see Figure~\ref{fig:DM1}).
 \end{lemma}
 
 \begin{proof}
 By Lemma~\ref{basic1}, $b_1 \neq d_1$. But $b_1=d_1$ by Lemma~\ref{basic2}, a contradiction.
 \end{proof}
 
 \begin{lemma}
  Let $a,b,c,d$ be distinct vertices. The algebraic coloring $\chi$ forbids $\chi(ab)=\chi(ac)=\chi(ad)$ and $\chi(bc)=\chi(bd)$ (see Figure~\ref{fig:DM2}).
 \end{lemma}
 
 \begin{proof}
 By Lemma~\ref{basic2}, $a_1=b_1$. Therefore, $\delta(a_1,b_1)=0$. So $c_1=a_1=d_1$. But $c_1 \neq d_1$ by Lemma~\ref{basic1}.
 \end{proof}
 
 Now we will take care of a few additional configurations. We will use the following technical lemma.
 
 \begin{lemma}
 \label{altC4lemma}
 Let $a,b,c,d$ be four distinct vertices such that $\chi(ab)=\chi(cd)$ and $\chi(bc)=\chi(ad)$. Then \[(a_1+c_1)(b_1-d_1) = 2(b_2-d_2).\]
 \end{lemma}
 
 \begin{proof}
 The two colors give us the following relations:
 \begin{align*}
 a_1b_1 - a_2 - b_2 &= c_1d_1 - c_2 - d_2\\
 a_1d_1 - a_2 - d_2 &= c_1b_1 - b_2 - b_2.
 \end{align*}
 We subtract the second equation from the first to get the desired equation.
 \end{proof}
 
\begin{figure}
     \begin{subfigure}[b]{0.2\textwidth}
          \centering
          \resizebox{\linewidth}{!}{
          \begin{tikzpicture}
		\draw[very thick, black] (1.0, 0.0) -- (0, 1.0);
		\draw[very thick, black] (0, 1.0) -- (-1.0, 0);
		\draw[very thick, black] (-1.0, 0) -- (0, -1.0);
		\draw[very thick, black] (0, -1) -- (1, 0);

		\filldraw [black] (1.0, 0.0) circle (1pt);
		\node[right] at (1,0) {$a$};
		\filldraw [black] (0, 1.0) circle (1pt);
		\node[above] at (0,1) {$b$};
		\filldraw [black] (-1.0, 0) circle (1pt);
		\node[left] at (-1,0) {$c$};
		\filldraw [black] (0, -1.0) circle (1pt);
		\node[below] at (0,-1) {$d$};
		
	\end{tikzpicture}
					}
          \caption{}
          \label{fig:DM1}
     \end{subfigure}
     \begin{subfigure}[b]{0.2\textwidth}
          \centering
          \resizebox{\linewidth}{!}{
        \begin{tikzpicture}
		\draw[very thick, red] (0, 1) -- (-1, 0);
		\draw[very thick, red] (0, 1) -- (0, -1);
		\draw[very thick, black] (1.0, 0.0) -- (0, 1.0);
		\draw[very thick, black] (1, 0) -- (-1.0, 0);
		\draw[very thick, black] (1, 0) -- (0, -1.0);

		\filldraw [black] (1.0, 0.0) circle (1pt);
		\node[right] at (1,0) {$a$};
		\filldraw [black] (0, 1.0) circle (1pt);
		\node[above] at (0,1) {$b$};
		\filldraw [black] (-1.0, 0) circle (1pt);
		\node[left] at (-1,0) {$c$};
		\filldraw [black] (0, -1.0) circle (1pt);
		\node[below] at (0,-1) {$d$};
		
	\end{tikzpicture}
          }
          \caption{}
          \label{fig:DM2}
     \end{subfigure}
\begin{subfigure}[b]{0.2\textwidth}
          \centering
          \resizebox{\linewidth}{!}{
\begin{tikzpicture}
\draw[very thick, black] (1.0, 0.0) -- (0.30901699437494745, 0.95105651629515353);
\draw[very thick, red] (0.30901699437494745, 0.95105651629515353) -- (-0.80901699437494734, 0.58778525229247325);
\draw[very thick, black] (-0.80901699437494734, 0.58778525229247325) -- (-0.80901699437494745, -0.58778525229247303);
\draw[very thick, red] (1.0, 0.0) -- (-0.80901699437494745, -0.58778525229247303);

\draw[very thick, blue] (0.30901699437494723, -0.95105651629515364) -- (0.30901699437494745, 0.95105651629515353);
\draw[very thick, blue] (0.30901699437494723, -0.95105651629515364) -- (-0.80901699437494745, -0.58778525229247303);

\draw[very thick, green] (0.30901699437494723, -0.95105651629515364) -- (1.0, 0.0);
\draw[very thick, green] (0.30901699437494723, -0.95105651629515364) -- (-0.80901699437494734, 0.58778525229247325);

\filldraw [black] (1.0, 0.0) circle (1pt);
\node[right] at (1,0) {$a$};

\filldraw [black] (0.30901699437494745, 0.95105651629515353) circle (1pt);
\node[above] at (0.30901699437494745, 0.95105651629515353) {$b$};

\filldraw [black] (-0.80901699437494734, 0.58778525229247325) circle (1pt);
\node[left] at (-0.80901699437494734, 0.58778525229247325) {$c$};

\filldraw [black] (-0.80901699437494745, -0.58778525229247303) circle (1pt);
\node[left] at (-0.80901699437494745, -0.58778525229247303) {$d$};

\filldraw [black] (0.30901699437494723, -0.95105651629515364) circle (1pt);
\node[below] at (0.30901699437494723, -0.95105651629515364) {$e$};
\end{tikzpicture}
}
\caption{}
\label{fig:C4alg1}
\end{subfigure}
\begin{subfigure}[b]{0.2\textwidth}
          \centering
          \resizebox{\linewidth}{!}{
\begin{tikzpicture}
\draw[very thick, black] (1.0, 0.0) -- (0.30901699437494745, 0.95105651629515353);
\draw[very thick, red] (0.30901699437494745, 0.95105651629515353) -- (-0.80901699437494734, 0.58778525229247325);
\draw[very thick, black] (-0.80901699437494734, 0.58778525229247325) -- (-0.80901699437494745, -0.58778525229247303);
\draw[very thick, red] (1.0, 0.0) -- (-0.80901699437494745, -0.58778525229247303);

\draw[very thick, blue] (0.30901699437494723, -0.95105651629515364) -- (1,0);
\draw[very thick, blue] (-0.80901699437494734, 0.58778525229247325) -- (1,0);

\draw[very thick, red] (0.30901699437494723, -0.95105651629515364) -- (0.30901699437494745, 0.95105651629515353);
\draw[very thick, red] (0.30901699437494723, -0.95105651629515364) -- (-0.80901699437494745, -0.58778525229247303);

\filldraw [black] (1.0, 0.0) circle (1pt);
\node[right] at (1,0) {$a$};

\filldraw [black] (0.30901699437494745, 0.95105651629515353) circle (1pt);
\node[above] at (0.30901699437494745, 0.95105651629515353) {$b$};

\filldraw [black] (-0.80901699437494734, 0.58778525229247325) circle (1pt);
\node[left] at (-0.80901699437494734, 0.58778525229247325) {$c$};

\filldraw [black] (-0.80901699437494745, -0.58778525229247303) circle (1pt);
\node[left] at (-0.80901699437494745, -0.58778525229247303) {$d$};

\filldraw [black] (0.30901699437494723, -0.95105651629515364) circle (1pt);
\node[below] at (0.30901699437494723, -0.95105651629515364) {$e$};
\end{tikzpicture}
}
\caption{}
\label{fig:C4alg2}
\end{subfigure}

\begin{subfigure}[b]{0.2\textwidth}
          \centering
          \resizebox{\linewidth}{!}{
\begin{tikzpicture}
\draw[very thick, black] (1.0, 0.0) -- (0.30901699437494745, 0.95105651629515353);
\draw[very thick, red] (0.30901699437494745, 0.95105651629515353) -- (-0.80901699437494734, 0.58778525229247325);
\draw[very thick, black] (-0.80901699437494734, 0.58778525229247325) -- (-0.80901699437494745, -0.58778525229247303);
\draw[very thick, red] (1.0, 0.0) -- (-0.80901699437494745, -0.58778525229247303);

\draw[very thick, red] (0.30901699437494723, -0.95105651629515364) -- (0.30901699437494745, 0.95105651629515353);
\draw[very thick, black] (0.30901699437494723, -0.95105651629515364) -- (-0.80901699437494745, -0.58778525229247303);

\filldraw [black] (1.0, 0.0) circle (1pt);
\node[right] at (1,0) {$a$};

\filldraw [black] (0.30901699437494745, 0.95105651629515353) circle (1pt);
\node[above] at (0.30901699437494745, 0.95105651629515353) {$b$};

\filldraw [black] (-0.80901699437494734, 0.58778525229247325) circle (1pt);
\node[left] at (-0.80901699437494734, 0.58778525229247325) {$c$};

\filldraw [black] (-0.80901699437494745, -0.58778525229247303) circle (1pt);
\node[left] at (-0.80901699437494745, -0.58778525229247303) {$d$};

\filldraw [black] (0.30901699437494723, -0.95105651629515364) circle (1pt);
\node[below] at (0.30901699437494723, -0.95105651629515364) {$e$};
\end{tikzpicture}
}
\caption{}
\label{fig:C4alg3}
\end{subfigure}
\begin{subfigure}[b]{0.2\textwidth}
          \centering
          \resizebox{\linewidth}{!}{
\begin{tikzpicture}
\draw[very thick, black] (1.0, 0.0) -- (0.30901699437494745, 0.95105651629515353);
\draw[very thick, red] (0.30901699437494745, 0.95105651629515353) -- (-0.80901699437494734, 0.58778525229247325);
\draw[very thick, red] (-0.80901699437494734, 0.58778525229247325) -- (-0.80901699437494745, -0.58778525229247303);
\draw[very thick, black] (1.0, 0.0) -- (-0.80901699437494745, -0.58778525229247303);

\draw[very thick, blue] (0.30901699437494723, -0.95105651629515364) -- (1,0);
\draw[very thick, blue] (-0.80901699437494734, 0.58778525229247325) -- (1,0);

\draw[very thick, black] (0.30901699437494723, -0.95105651629515364) -- (0.30901699437494745, 0.95105651629515353);
\draw[very thick, red] (0.30901699437494723, -0.95105651629515364) -- (-0.80901699437494745, -0.58778525229247303);

\filldraw [black] (1.0, 0.0) circle (1pt);
\node[right] at (1,0) {$a$};

\filldraw [black] (0.30901699437494745, 0.95105651629515353) circle (1pt);
\node[above] at (0.30901699437494745, 0.95105651629515353) {$b$};

\filldraw [black] (-0.80901699437494734, 0.58778525229247325) circle (1pt);
\node[left] at (-0.80901699437494734, 0.58778525229247325) {$c$};

\filldraw [black] (-0.80901699437494745, -0.58778525229247303) circle (1pt);
\node[left] at (-0.80901699437494745, -0.58778525229247303) {$d$};

\filldraw [black] (0.30901699437494723, -0.95105651629515364) circle (1pt);
\node[below] at (0.30901699437494723, -0.95105651629515364) {$e$};
\end{tikzpicture}
}
\caption{}
\label{fig:ALG1}
\end{subfigure}
\begin{subfigure}[b]{0.2\textwidth}
          \centering
          \resizebox{\linewidth}{!}{
\begin{tikzpicture}
\draw[very thick, black] (1.0, 0.0) -- (0.30901699437494745, 0.95105651629515353);
\draw[very thick, red] (0.30901699437494745, 0.95105651629515353) -- (-0.80901699437494734, 0.58778525229247325);
\draw[very thick, red] (-0.80901699437494734, 0.58778525229247325) -- (-0.80901699437494745, -0.58778525229247303);
\draw[very thick, black] (1.0, 0.0) -- (-0.80901699437494745, -0.58778525229247303);

\draw[very thick, black] (0.30901699437494723, -0.95105651629515364) -- (0.30901699437494745, 0.95105651629515353);
\draw[very thick, red] (0.30901699437494723, -0.95105651629515364) -- (-0.80901699437494745, -0.58778525229247303);

\draw[very thick, blue] (0.30901699437494723, -0.95105651629515364) -- (1,0);
\draw[very thick, blue] (0.30901699437494723, -0.95105651629515364) -- (-0.80901699437494734, 0.58778525229247325);

\filldraw [black] (1.0, 0.0) circle (1pt);
\node[right] at (1,0) {$a$};

\filldraw [black] (0.30901699437494745, 0.95105651629515353) circle (1pt);
\node[above] at (0.30901699437494745, 0.95105651629515353) {$b$};

\filldraw [black] (-0.80901699437494734, 0.58778525229247325) circle (1pt);
\node[left] at (-0.80901699437494734, 0.58778525229247325) {$c$};

\filldraw [black] (-0.80901699437494745, -0.58778525229247303) circle (1pt);
\node[left] at (-0.80901699437494745, -0.58778525229247303) {$d$};

\filldraw [black] (0.30901699437494723, -0.95105651629515364) circle (1pt);
\node[below] at (0.30901699437494723, -0.95105651629515364) {$e$};
\end{tikzpicture}
}
\caption{}
\label{fig:ALG2}
\end{subfigure}

\caption{Configurations eliminated by the modified algebraic coloring.}
\label{algconf}
\end{figure}
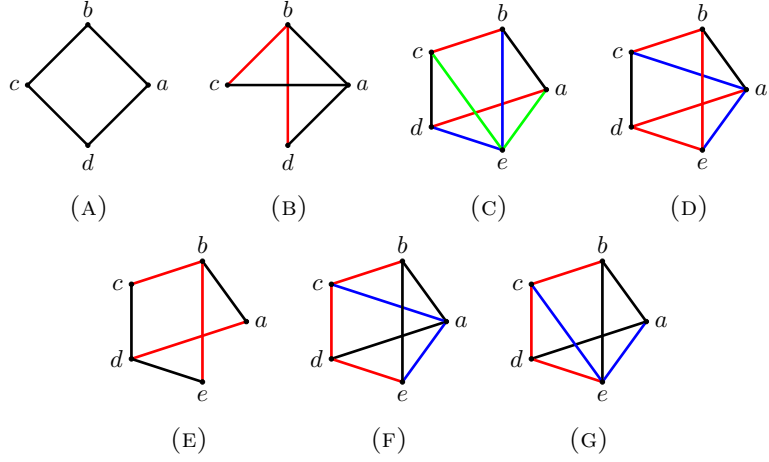

\begin{lemma}
Let $a,b,c,d,e$ be distinct vertices. Then $\mathcal{C}(ab)=\mathcal{C}(cd)$, $\mathcal{C}(bc)=\mathcal{C}(ad)$, $\mathcal{C}(ae)=\mathcal{C}(ce)$, and $\mathcal{C}(be)=\mathcal{C}(de)$ (see Figure~\ref{fig:C4alg1}) is forbidden by the coloring $\mathcal{C} = \varphi \times \chi$.
\end{lemma}

\begin{proof}
By Lemma~\ref{altC4lemma}, we know that \[(a_1+c_1)(b_1-d_1) = 2(b_2-d_2),\] and from $\chi(be)=\chi(de)$ we get that \[e_1(b_1-d_1) = b_2-d_2.\] Therefore,
\begin{align*}
(a_1+c_1)(b_1-d_1) &= 2e_1(b_1-d_1)\\
a_1+c_1 &= 2e_1
\end{align*}
since $b_1 \neq d_1$ by Lemma~\ref{basic1}. So $f_{e_1}(a_1) \neq f_{e_1}(c_1)$. By Lemma~\ref{order}, $\varphi(ae)=\varphi(ce)$ implies that either $a,c<e$ or $e<a,c$. In either case, $\chi_2(ae) \neq \chi_2(ce)$, a contradiction.
\end{proof}

\begin{lemma}
Let $a,b,c,d,e$ be distinct vertices. Then $\mathcal{C}(ab)=\mathcal{C}(cd)$, $\mathcal{C}(bc)=\mathcal{C}(ad) = \mathcal{C}(be) = \mathcal{C}(de)$, and $\mathcal{C}(ac)=\mathcal{C}(ae)$ (see Figure~\ref{fig:C4alg2}) is forbidden by the coloring $\mathcal{C} = \varphi \times \chi$.
\end{lemma}

\begin{proof}
As in the previous proof we get that $a_1+c_1 = 2e_1$. By Lemma~\ref{basic2} we know that $b_1 = a_1$. Therefore, it follows from the second part of $\chi_1$ that $c_1=d_1$. Therefore, $b_1+d_1 = 2e_1$ and so $f_{e_1}(b_1) \neq f_{e_1}(d_1)$. As before, this fact along with Lemma~\ref{order} forces $\chi_2(be) \neq \chi_2(de)$, a contradiction.
\end{proof}

\begin{lemma}
Let $a,b,c,d,e$ be distinct vertices. Then $\chi(ab)=\chi(cd) = \chi(de)$ and $\chi(bc)=\chi(ad) = \chi(be)$ (see Figure~\ref{fig:C4alg3}) is forbidden by the algebraic coloring $\chi$.
\end{lemma}

\begin{proof}
By Lemma~\ref{altC4lemma} we know that \[(a_1+c_1)(b_1-d_1) = 2(b_2-d_2),\] and by Lemma~\ref{basic2} we know that $b_1=d_1$. Hence, $b_2-d_2=0$ and so $b=d$, a contradiction.
\end{proof}
 
 \begin{lemma}
Let $a,b,c,d,e$ be distinct vertices. Then $\chi(bc)=\chi(cd) = \chi(de)$, $\chi(eb)=\chi(ba)=\chi(ad)$, and $\chi(ac)=\chi(ae)$ (see Figure~\ref{fig:ALG1}) is forbidden by the algebraic coloring $\chi$.
\end{lemma}

\begin{proof}
By Lemma~\ref{basic1} we get that $b_1 \neq d_1$. By Lemma~\ref{basic2} we get that $a_1=d_1$. Therefore, since the color encodes equality in the first coordinate we see that $b_1=d_1$, a contradiction.
\end{proof}

\begin{lemma}
Let $a,b,c,d,e$ be distinct vertices. Then $\chi(bc)=\chi(cd) = \chi(de)$, $\chi(eb)=\chi(ba)=\chi(ad)$, and $\chi(ec)=\chi(ae)$ (see Figure~\ref{fig:ALG2}) is forbidden by the algebraic coloring $\chi$.
\end{lemma}

\begin{proof}
By Lemma~\ref{basic2} we get that $a_1=c_1$. By Lemma~\ref{basic1}, we get that $a_1 \neq c_1$, a contradiction.
\end{proof}

\appendix

\section{Algorithm for reducing cases}
\label{script}
The following algorithm is not difficult to verify, so we present it here without proof. The specific implementation we rely on is a Python script that can be found (with comments) at \url{http://homepages.math.uic.edu/~acamer4/EdgeColors56.py}. This particular script forbids monochromatic odd cycles and the edge-colorings shown in Figures~\ref{colorcycles},~\ref{TheP3s},~\ref{altC4s},~\ref{additional}, as well as Figures~\ref{fig:DM1} and~\ref{fig:DM2}. The output is seven edge-colored copies of $K_5$ that each contain one of the remaining configurations in Figure~\ref{algconf}.

Suppose we want to find every edge-coloring, up to isomorphism, of $K_n$ that uses at most $m$ colors and does not contain a copy of any $F \in \mathcal{F}$, a list of edge-colored complete graphs on $n$ or fewer vertices. The algorithm takes $\mathcal{F}$, $n$, and $m$ as input and returns a list $\mathcal{R}$ of edge-colorings of $K_n$ satisfying these requirements.

For each $k=3,\ldots,n$, the algorithm creates a list $L_k$ of acceptable edge-colorings of $K_k$ by adding a new vertex to each $K_{k-1}$ listed in $L_{k-1}$ (where $L_2$ is the list of exactly one $K_2$ with its single edge given color $1$), and then coloring the $k-1$ new edges in all possible ways from the color set $[m]$. For each graph in $L_{k-1}$ and each way to color the new edges, we test the resulting graph to see if it contains any of the forbidden edge-colorings. If it does, then we move on. If not, then we test it against the new list $L_k$ to see if it is isomorphic to any of the colorings of $K_k$ already on the list. If it is, then we move on. Otherwise, we add it to the list $L_k$. The algorithm terminates when it has tested all colorings of $K_n$.\\

\begin{algorithm}[H]
\SetAlgoLined
\KwData{number of vertices $n$; maximum number of colors $m$; list of forbidden colorings $\mathcal{F}$}
 
initialize $L_2$ as list containing one $K_2$ with its edges colored 1\;
\For{$k=3,\ldots,n$}{
initialize empty list $L_k$\;
\For{$H \in L_{k-1}$}{
\For{each function $f:[k-1] \rightarrow [m]$}{
let $G$ be $K_{k}$ with edge-colors same as $H$ on the first $k-1$ vertices and color $f(i)$ on edge $ki$ for $i=1,\ldots,k-1$\;
\If{$G$ contains no element of $\mathcal{F}$ and is isomorphic to no element of $L_k$}{add $G$ to the list $L_k$}
}
}
}
\Return{$L_n$}
 \caption{List all edge-colorings with no forbidden subcoloring}
\end{algorithm}

\bibliography{56construction}
\bibliographystyle{plain}

\end{document}